\documentclass[12pt]{amsproc}

\usepackage[T1]{fontenc}
\usepackage[utf8]{inputenc}
\usepackage{amsmath,amssymb,amsthm,mathtools}
\usepackage{enumitem}
\usepackage[backref]{hyperref}
\usepackage[
  a4paper,
  left=27mm,
        right=27mm,
  top=27mm,
  bottom=30mm
]{geometry}

\hypersetup{
  colorlinks=true,
  linkcolor=blue,
  citecolor=blue,
  urlcolor=blue
}

\newtheorem{maintheorem}{Theorem}

\newtheorem{theorem}{Theorem}[section]
\newtheorem{proposition}[theorem]{Proposition}
\newtheorem{lemma}[theorem]{Lemma}
\newtheorem{corollary}[theorem]{Corollary}
\newtheorem{conjecture}[theorem]{Conjecture}
\newtheorem{openproblem}[theorem]{Open problem}

\theoremstyle{definition}
\newtheorem{definition}[theorem]{Definition}

\theoremstyle{remark}
\newtheorem{remark}[theorem]{Remark}

\newcommand{\cU}{\mathcal U}
\newcommand{\cS}{\mathcal S}
\newcommand{\cC}{\mathcal C}
\newcommand{\Sym}{\operatorname{Sym}}
\newcommand{\dist}{\operatorname{dist}}
\newcommand{\Nm}{\operatorname{N}}

\newcommand{\Prob}{\mathbb P}

\title[Centralizers of sofic approximations]{Centralizers of sofic approximations \\ of Kazhdan groups}

\author{Vadim Alekseev}
\address{Vadim Alekseev, TU Dresden, 01062 Dresden, Germany}
\email{vadim.alekseev@tu-dresden.de}

\author{Andreas Thom}
\address{Andreas Thom, TU Dresden, 01062 Dresden, Germany}
\email{andreas.thom@tu-dresden.de}

\date{}


\begin{document}

\begin{abstract}
We prove that a Kazhdan group admitting a sofic embedding into a metric ultraproduct of symmetric groups
with a centralizer that acts ergodically on the associated Loeb probability space is locally embeddable in finite groups (LEF). In particular, every finitely presented Kazhdan group admitting such an embedding is residually finite. The main technical theorem says that the centralizer of a sofic embedding of a Kazhdan group is itself a metric ultraproduct of permutation groups.
\end{abstract}

\maketitle

\tableofcontents

\section{Introduction}

The class of sofic groups was introduced by Gromov \cite{GromovSymbolic} in connection with his seminal work on Gottschalk's surjunctivity conjecture and subsequently named and more
systematically studied by Weiss \cite{WeissSofic}. It contains
all amenable groups and all residually finite groups, and no non-sofic group is
presently known. From the viewpoint of analytic group theory, a plausible place
to look for a non-sofic group is among finitely presented, non-residually
finite Kazhdan groups\footnote{We finished writing this paper mid July 2026 and it was circulated among some colleagues. On August 1, 2026, OpenAI has announced a proof of the existence of a non-sofic group, see [OpenAI,
\emph{A counterexample to the soficity conjecture},
in \emph{Ten advances in mathematics and theoretical computer science},
Chapter~3, 77--93, 2026,
\url{https://cdn.openai.com/pdf/ten-proofs-oai.pdf}]. This group can be taken finitely presented and Kazhdan. Since the paper was finished, when this announcement was made, we decided to keep the presentation of this paper unchanged.}. The purpose of this paper is to study sofic approximations of such
groups. Our analysis is ultimately based on Kun's seminal work \cite{Kun} and later
developments based on it, including
\cite{KunThom,AlekseevThom,AlekseevThomApprox}.

Universal sofic groups are metric ultraproducts of finite symmetric groups with
normalized Hamming metrics
$
        \cS_{\mathcal U}:=\prod_{\cU}\Sym(n),
$
where \(\cU\) is a non-principal ultrafilter, see \cite{ElekSzabo}. It is well-known that a countable group is sofic precisely when it embeds
into such a universal sofic group. The metric ultraproduct acts on the associated Loeb
probability space $(X,\mu):=\prod_{\mathcal U} (\{1,\dots,n\},\mu_n)$, where $\mu_n$ denotes the normalized counting measure, and one may therefore study sofic approximations one by one in terms of the ergodic properties of the limit action on $(X,\mu)$. This viewpoint has already been proven fruitful in \cite{GohlaThom}. In particular, we may consider sofic embeddings whose
centralizer inside the universal sofic group has prescribed ergodic properties. Hayes and Kunnawalkam Elayavalli
formulated the following conjecture in their work on non-conjugacy of sofic
approximations \cite{HayesKunnawalkam}.

\begin{conjecture}[Hayes--Kunnawalkam Elayavalli]
Every sofic group \(G\) admits an embedding \(\pi:G\to \cS_{\mathcal U}\)
into a universal sofic group such that \(C_{\cS_{\mathcal U}}(\pi(G))\) acts
ergodically on the associated Loeb probability space.
\end{conjecture}

For a large class of sofic groups the conjectured embedding is known to exist.
Indeed, this is the case for initially subamenable groups by
 work of P\u aunescu \cite{Paunescu,PaunescuConvex}, following Kerr--Li
\cite{KerrLi}, as explained in \cite{HayesKunnawalkam}. The
main result of this paper shows however that for Kazhdan groups,
the additional ergodicity requirement on the centralizer is very restrictive.

\begin{maintheorem}\label{thm:main}
Let \(G\) be a Kazhdan group. Suppose that \(G\) admits a
sofic embedding \(\pi:G\to \cS_{\mathcal U}\) such that the centralizer
\(C_{\cS_{\mathcal U}}(\pi(G))\) acts ergodically on the associated Loeb
probability space. Then \(G\) is locally embeddable into finite groups. In particular, if \(G\) is finitely
presented, then \(G\) is residually finite.
\end{maintheorem}

The proof is organized as follows. As a preliminary step, we first compute almost centralizers of
finite transitive actions by a majority argument. We then show that if the
centralizer of a sofic embedding is isomorphic to a metric ultraproduct of
finite permutation groups and acts ergodically, then finite generation forces the group to be
locally embeddable into finite groups (LEF). If the group is finitely presented, this implies that the group is residually finite. 
The technical
heart of the proof is the description of the centralizer of the sofic embedding in the universal sofic group in terms of the finite-level approximate centralizers. Indeed, Theorem \ref{thm:groupoid-recovery} says that after inessential changes of the finite models, the centralizer of the sofic approximation  $(\theta_n \colon \Gamma \to \Sym(X_n))_n$ is a metric ultraproduct of finite permutation groups $A_n \subset \Sym(X_n)$.
This is achieved as a groupoid version of the results obtained by Kun and the second author \cite{KunThom}. Indeed, after
Kun's expander decomposition, a repair argument gives a well-defined composition for
the finite cluster groupoids of almost equivariant partial bijections, and the
Becker--Chapman stability theorem \cite{BeckerChapman} corrects the approximate actions to genuine actions of the full groups of those groupoids. Proposition~\ref{prop:algebraic-recovery-rf} then converts
this centralizer structure into an LEF approximation.
These ingredients prove Theorem~\ref{thm:main}.

\section{Definitions and preliminaries}\label{sec:prelim}

\subsection{Definitions}

All finite sets are equipped with normalized counting measure. For permutations of a finite set \(X\), the
normalized Hamming distance is
\[
        d_H(\sigma,\tau)=\frac{|\{x\in X:\sigma x\neq\tau x\}|}{|X|}.
\]
Along an ultrafilter \(\cU\), we write \(a_n\to_{\cU}a\) for a sequence $(a_n)$ of real numbers and $a \in \mathbb R$ for the usual notion of convergence along the ultrafilter.

\begin{definition}
Let \(\Lambda\) be a finitely generated group. A sequence of maps
\(\theta_n:\Lambda\to\Sym(W_n)\) for finite sets \(W_n\) is a \emph{sofic approximation} if, for all
\[
        d_H(\theta_n(\lambda\mu),\theta_n(\lambda)\theta_n(\mu))\to0,
\]
and, for every \(\lambda\neq e\),
$
        d_H(\theta_n(\lambda),1)\to1.
$
\end{definition}

It is frequently convenient to represent a sofic approximation of a finitely generated group \(\Lambda\) by a sequence of homomorphisms from a fixed finitely generated free group \(F\) that maps onto \(\Lambda\). Let \(S=S^{-1}\) be a finite generating set of \(\Lambda\), and let \(F=F_S\) be the free group on the alphabet \(S\). Then a sofic approximation of \(\Lambda\) is represented by a sequence of homomorphisms \(\theta'_n :F\to\Sym(W_n)\) such that, for every \(w\in F\),
\[
        \lim_{n \to \infty} d_H(\theta'_n(w),1)=
        \begin{cases}
        0 & \text{if } w \in \ker(F \to \Lambda),\\
        1 & \text{otherwise}.
        \end{cases}
\]

We recall that a finitely generated group \(G\) is \emph{locally embeddable
into finite groups} (LEF) if, for every finite set \(E\subseteq G\), there are
a finite group \(K\) and an injective map \(\theta:E\to K\) such that
\(\theta(xy)=\theta(x)\theta(y)\) whenever \(x,y,xy\in E\). It is easy to see that a finitely generated LEF group is sofic.

\begin{definition}
Let \(T=T^{-1}\) be finite and suppose that permutations \(\theta(t)\),
\(t\in T\), of a finite set \(W\) are given. For \(A\subseteq W\), the
\emph{edge boundary} is
\[
        \partial_T A
        =\bigl\{(x,t)\in A\times T:\theta(t)x\notin A\bigr\},
\]
and \((W,\theta)\) is an \emph{\(h\)-expander}, for \(h>0\),
if
$$|\partial_T A|\ge h\min\bigl(|A|,|W\setminus A|\bigr)$$ for every 
\(A\subseteq W\).  Let $\Lambda$ be a finitely generated group generated by a finite set $S$. A sofic approximation \(\theta_n:\Lambda\to\Sym(W_n)\) is an \emph{expander
sofic approximation}, with expander constant \(h>0\), if every \((W_n,\theta_n(S)   )  \) is an \(h\)-expander. 
\end{definition}

We record the two elementary reformulations used below; both are immediate
from the definitions and hold for every \(A\subseteq W\). With normalized
counting measure \(\nu\), the symmetrized generator count is twice the edge
boundary,
\begin{equation}
        \sum_{t\in T}\nu\{x:\mathbf 1_A(\theta(t)x)\neq\mathbf 1_A(x)\}
        =\frac{2\,|\partial_T A|}{|W|},
        \label{eq:boundary-symmetrize}
\end{equation}
because each pair \((x,t)\) with \(\theta(t)x\notin A\), \(x\in A\), is matched
by the pair \((\theta(t)x,t^{-1})\) with \(t^{-1}\)-neighbour in \(A\). In
particular, an \(h\)-expander satisfies the symmetrized bound
\[
        \sum_{t\in T}\nu\{x:\mathbf 1_A(\theta(t)x)\neq\mathbf 1_A(x)\}
        \ge 2h\min\bigl(\nu(A),1-\nu(A)\bigr).
\]

We will frequently compare sofic approximations and make negligible modifications on both the permutations involved and the ground sets. The following definition formalizes these notions:

\begin{definition}
\label{def:essentially}
Two sofic approximations of a finitely generated group $\Lambda$ on \(X_n\) and \(X'_n\) are \emph{essentially equivalent} if
there are subsets \(U_n\subseteq X_n\) and \(U'_n\subseteq X'_n\), with
\(|X_n\setminus U_n|=o(|X_n|)\) and
\(|X'_n\setminus U'_n|=o(|X'_n|)\), and bijections \(U_n\to U'_n\) under which
all generator maps agree.
Along \(\cU\) we use the same convention with \(o(1)\) replaced by convergence
to zero along \(\cU\). If a sofic approximation has a property after replacement by an
essentially equivalent model, we say that it has that property \emph{essentially}.
\end{definition}
The following lemma is standard:
\begin{lemma}
\label{lem:null-modification}
Let \(\Lambda\) be generated by a finite symmetric set \(R\), and let
\(\theta_n:F_R\to\Sym(X_n)\) and \(\theta'_n:F_R\to\Sym(X'_n)\) be two sofic approximations of $\Lambda$ which are essentially equivalent. Then the following holds:
\begin{enumerate}[label=\textup{(\roman*)}]
\item The induced metric ultraproducts and Loeb spaces are canonically isomorphic modulo null sets, and the actions are intertwined by this identification. 
\item The induced sofic embeddings are canonically conjugate. 
\item Under this identification, centralizers in the metric
ultraproduct are identified and ergodicity of their actions on the Loeb space are preserved.
\end{enumerate}
\end{lemma}

We use the following form of Kun's expander decomposition theorem
\cite{Kun}; see also the discussion in \cite{KunThom}.

\begin{theorem}[Kun]\label{thm:kun}
Let \(K\) be a Kazhdan group with finite symmetric
generating set \(R\). There is a constant \(h_K>0\), depending only on \(K\)
and \(R\), with the following property. For every sofic approximation
\(\theta_n:F_R\to\Sym(W_n)\) of \(K\) there are homomorphisms
\(\theta_n^\sharp:F_R\to\Sym(W_n)\) and \(\theta_n^\sharp\)-invariant subsets
\(W'_n\subseteq W_n\) such that:
\begin{enumerate}[label=\textup{(\roman*)}]
\item \(d_H(\theta_n^\sharp(r),\theta_n(r))\to0\) for every \(r\in R\);
\item \(|W_n\setminus W'_n|=o(|W_n|)\), and every point of
\(W_n\setminus W'_n\) is fixed by \(\theta_n^\sharp(r)\) for every \(r\in R\);
\item every connected component of the \(R\)-labelled graph of
\(\theta_n^\sharp\) on \(W'_n\) has Cheeger constant at least \(h_K\).
\end{enumerate}
In particular, every connected component of the \(R\)-labelled graph of
\(\theta_n^\sharp\) on all of \(W_n\) is either a singleton fixed by all
generators or has Cheeger constant at least \(h_K\), and
\(\theta_n^\sharp\) is essentially equivalent to \(\theta_n\).
\end{theorem}

\subsection{Uniform almost centralizers of transitive actions}\label{sec:majority}

If a finite group \(H\) acts on \(X\), we write \(C_{\Sym(X)}(H)\) for the exact centralizer of the image of
\(H\) in \(\Sym(X)\). The following lemma shows that the centralizer of a transitive group action is rigid with respect to the normalized Hamming metric.
\begin{lemma}\label{lem:majority}
Let \(H\) be a finite group and let \(L\leq H\).  Put \(X=H/L\), with
\(H\) acting on \(X\) by left translations.  Let \(\sigma\in \Sym(X)\) and
assume that \(\sup_{h\in H} d_H(\sigma h,h\sigma)\leq \varepsilon\).
If \(\varepsilon<1/2\), then there exists \(a\in \Nm_H(L)\) such that
\(d_H\bigl(\sigma,\rho(aL)\bigr)\leq \varepsilon\), where
\(\rho(aL)(gL)=gaL\) is the corresponding element of
\(C_{\Sym(H/L)}(H)\cong \Nm_H(L)/L\).  In particular,
\(\dist\bigl(\sigma,C_{\Sym(H/L)}(H)\bigr)\leq \varepsilon\).
\end{lemma}

\begin{proof}
Let \(o=L\in H/L\).  For \(g\in H\), define
\(c(g)=g^{-1}\sigma(go)\in H/L\).  For \(h,g\in H\), the equality
\(\sigma(hgo)=h\sigma(go)\) is equivalent to \(c(hg)=c(g)\).  Indeed,
\(c(hg)=(hg)^{-1}\sigma(hgo)=g^{-1}h^{-1}\sigma(hgo)\), and this equals
\(g^{-1}\sigma(go)=c(g)\) precisely when
\(\sigma(hgo)=h\sigma(go)\).

Hence the assumption gives, for every fixed
\(h\in H\), that
$$|\{g\in H:c(hg)\neq c(g)\}|/|H|\leq \varepsilon.$$
Averaging over \(h\) and using that, for every pair \((g,k)\in H^2\), there
is a unique \(h=kg^{-1}\) with \(k=hg\), we obtain
\(\Prob_{g,k\in H}(c(k)\neq c(g))\leq \varepsilon\), or equivalently
\(\Prob_{g,k\in H}(c(k)=c(g))\geq 1-\varepsilon\).  If
\(p_y=|\{g\in H:c(g)=y\}|/|H|\) for \(y\in H/L\), then
\(\sum_{y\in H/L} p_y^2\geq 1-\varepsilon\).
Since \(\sum_y p_y^2\leq \max_y p_y\), there exists \(aL\in H/L\) such
that \(p_{aL}\geq 1-\varepsilon\).  Thus \(c(g)=aL\) for all but at most
\(\varepsilon |H|\) elements \(g\in H\).

It remains to prove that \(a\) normalizes \(L\).  For \(\ell\in L\), since
\(\ell o=o\), one has
$$c(g\ell)=(g\ell)^{-1}\sigma(g\ell o)
=\ell^{-1}g^{-1}\sigma(go)=\ell^{-1}c(g).$$
Right multiplication by \(\ell\) preserves the uniform measure on \(H\), so
the distribution of \(c(g)\) is invariant under the left action of \(L\) on
\(H/L\).  Since \(p_{aL}>1/2\), the atom \(aL\) is the unique atom of mass
larger than \(1/2\).  Therefore \(\ell^{-1}aL=aL\) for every \(\ell\in L\),
that is, \(La\subseteq aL\), so \(a^{-1}La\subseteq L\), and since \(L\) is
finite this forces \(a^{-1}La=L\), i.e.\ \(a\in \Nm_H(L)\).

For this \(a\), the map \(\rho(aL)\) is well-defined and centralizes the left
action of \(H\).  Moreover, whenever \(c(g)=aL\), we have
\(\sigma(gL)=gaL=\rho(aL)(gL)\).
Since the condition \(c(g)=aL\) is constant on right cosets of \(L\), the
same estimate holds on \(H/L\).  Hence
\(d_H\bigl(\sigma,\rho(aL)\bigr)\leq \varepsilon\).
\end{proof}

\begin{proposition}\label{prop:centralizer-transitive}
Let \((H_n)_n\leq \Sym(X_n)\) be a sequence of finite permutation groups acting transitively on \(X_n\cong H_n/L_n\).  Let \(\cU\) be a nonprincipal ultrafilter.  Then, inside
\(\prod_{\cU}\Sym(X_n)\), one has
\[
        C\left(\prod_{\cU}H_n\right)
        =
        \prod_{\cU} C_{\Sym(X_n)}(H_n)
        =
        \prod_{\cU}\Nm_{H_n}(L_n)/L_n .
\]
All ultraproducts in this display are metric ultraproducts for the normalized
Hamming metrics induced by actions on \(X_n\).
\end{proposition}

\begin{proof}
The inclusion from right to left is immediate.  Conversely, let
\(\sigma=[\sigma_n]_{\cU}\in \prod_{\cU}\Sym(X_n)\) commute with every
element of \(\prod_{\cU}H_n\).  We claim that
\(\sup_{h\in H_n}d_H(\sigma_n h,h\sigma_n)\to_{\cU}0\).
If not, then for some \(\delta>0\), on a \(\cU\)-large set one can choose
\(h_n\in H_n\) such that \(d_H(\sigma_nh_n,h_n\sigma_n)>\delta\).
Then \([h_n]_{\cU}\in \prod_{\cU}H_n\) does not commute with
\([\sigma_n]_{\cU}\), a contradiction.

Lemma~\ref{lem:majority} gives \(b_n\in C_{\Sym(X_n)}(H_n)\) such that
\(d_H(\sigma_n,b_n)\to_{\cU}0\).
Thus \([\sigma_n]_{\cU}=[b_n]_{\cU}\), proving the reverse inclusion.
\end{proof}

\subsection{Residual finiteness from the centralizer structure}\label{sec:abstract}

Let \(\pi\colon G\to \prod_{\cU}\Sym(X_n)\) be a sofic embedding.  We say that the
centralizer of \(\pi(G)\) is \emph{rigid} if there are  subgroups \(A_n\leq \Sym(X_n)\) such that
\(C(\pi(G))=\prod_{\cU}A_n\) inside \(\prod_{\cU}\Sym(X_n)\), where this is a
metric ultraproduct with the normalized Hamming metrics inherited from the
actions on \(X_n\).

\begin{lemma}\label{lem:conull-orbit}
Assume that
\(A=\prod_{\cU}A_n\leq \prod_{\cU}\Sym(X_n)\) acts ergodically on the Loeb
probability space of \((X_n)_n\).  Then, after passing to a \(\cU\)-large set
of indices, there are \(A_n\)-orbits \(O_n\subseteq X_n\) such that
\(|O_n|/|X_n|\to_{\cU}1\).
\end{lemma}

\begin{proof}
Suppose no such conull orbit exists.  Then there are \(\delta>0\) and a
\(\cU\)-large set of indices such that every \(A_n\)-orbit in \(X_n\) has
measure at most \(1-\delta\).  For those \(n\), choose a union \(Y_n\) of
\(A_n\)-orbits with \(\delta/2\leq |Y_n|/|X_n|\leq 1-\delta/2\).
This is done greedily: add orbits until the measure first exceeds
\(\delta/2\).  Since each orbit has measure at most \(1-\delta\), the resulting
union has measure at most
$
        \frac{\delta}{2}+(1-\delta)=1-\frac{\delta}{2}.
$
Defining \(Y_n\) arbitrarily outside
the chosen \(\cU\)-large set, the ultraproduct \([Y_n]_{\cU}\) is a non-null
and non-conull Loeb measurable set invariant under \(A\), contradicting
ergodicity.
\end{proof}

\begin{proposition}\label{prop:algebraic-recovery-rf}
Let \(G\) be finitely generated and let
\(\pi\colon G\to \prod_{\cU}\Sym(X_n)\) be a sofic embedding. Assume that
\(C(\pi(G))=\prod_{\cU}A_n\) for subgroups \(A_n\leq \Sym(X_n)\), and
assume that this centralizer acts ergodically on the associated Loeb space.
Then \(G\) is LEF. In particular, if \(G\) is finitely presented, then \(G\) is
residually finite.
\end{proposition}

\begin{proof}
By Lemma~\ref{lem:conull-orbit}, choose \(A_n\)-orbits \(O_n\subseteq X_n\)
with \(|O_n|/|X_n|\to_{\cU}1\). The identity maps on \(O_n\) exhibit the model
on \(X_n\) and its restriction to \(O_n\) as essentially equivalent. Thus, without loss of generality, we may assume that \( A_n\curvearrowright X_n\) are transitive.

Choose a basepoint in each \(X_n\). Then \(X_n=A_n/L_n\) for some subgroup
\(L_n\leq A_n\). Since \(\pi(G)\) centralizes
\(A=\prod_{\cU}A_n\), Proposition~\ref{prop:centralizer-transitive} gives
\[
        \pi(G)\subseteq C\left(\prod_{\cU}A_n\right)
        =
        \prod_{\cU}B_n,
        \qquad
        B_n:=\Nm_{A_n}(L_n)/L_n.
\]
Here \(\prod_{\cU}A_n\) is the metric ultraproduct associated with the
actions \(A_n\curvearrowright A_n/L_n\). The group \(B_n\) acts on \(A_n/L_n\)
by right translations. This action is free: if
\(bL_n\in \Nm_{A_n}(L_n)/L_n\) fixes \(aL_n\), then \(ab\in aL_n\), so
\(b\in L_n\), and hence the element is trivial. Therefore every nontrivial
element of \(B_n\) moves every point of \(X_n\). It follows that the metric
ultraproduct \(\prod_{\cU}B_n\), equipped with normalized Hamming metric, is discrete.

Thus, $G$ is identified with a subgroup of the algebraic ultraproduct \(\prod_{\cU}B_n\) and hence LEF. Moreover, if $G$ is finitely presented, then it is residually finite.
\end{proof}

It remains to show that the centralizer of a sofic embedding of a Kazhdan group is essentially rigid in the sense introduced above. This is the content of the next section.

\section{Approximate intertwiners for sofic approximations of Kazhdan groups}\label{sec:groupoid}

Throughout this section \(G\) is a Kazhdan group, \(S=S^{-1}\)
is a fixed finite generating set, \(F=F_S\), and
\(\pi:G\to \prod_{\cU}\Sym(X_n)\)
is a sofic embedding represented by homomorphisms
\(\alpha_n:F\to \Sym(X_n).\)

\begin{theorem}\label{thm:groupoid-recovery}
Let \(G\) be a Kazhdan group and let
\(\pi:G\to \prod_{\cU}\Sym(X_n)\)
be a sofic embedding represented by homomorphisms
\(\alpha_n:F_S\to\Sym(X_n)\). Then there exist an essentially equivalent  sofic embedding
\(\pi':G\to\prod_{\cU}\Sym(Y_n),\)
and subgroups
\(A_n\leq \Sym(Y_n)\)
such that
\[
        C_{\prod_{\cU}\Sym(Y_n)}(\pi'(G))
        =
        \prod_{\cU} A_n .
\]
\end{theorem}

\subsection{Good components and scale choice}

If \(Y,Z\) are finite \(S\)-labelled \(F\)-sets and
\(b:Y\dashrightarrow Z\) is a partial bijection, define
\[
        \lambda_Y(b)=\frac{|Y\setminus \operatorname{dom}(b)|}{|Y|},
        \qquad
        \lambda_Z(b)=\frac{|Z\setminus \operatorname{ran}(b)|}{|Z|},
\]
and
\[
        \Delta_S(b)
        =
        \frac{1}{|Y|}
        \left|
        \left\{
        (y,s)\in Y\times S:
        y\notin \operatorname{dom}(b)\ \text{or}\ sy\notin \operatorname{dom}(b)\ \text{or}\
        b(sy)\neq s b(y)
        \right\}
        \right|.
\]
If \(b,c:Y\dashrightarrow Z\) are partial maps, define
\[
        d_Y(b,c)
        =
        \frac{1}{|Y|}
        \left|
        \left\{
        y\in Y:
        y\notin \operatorname{dom}(b)\cap \operatorname{dom}(c)\ \text{or}\ b(y)\neq c(y)
        \right\}
        \right|.
\]
Undefined values are thus counted as disagreements.  On a component
\(X_{n,i}\), we write \(d_i=d_{X_{n,i}}\).

By Theorem~\ref{thm:kun}, we may assume that
\(X_n=\bigsqcup_{i\in I_n} X_{n,i}\)
is a decomposition into \(F\)-orbits whose labelled Schreier graphs have
Cheeger constant at least a fixed number \(h>0\).

\begin{lemma}\label{lem:good-components}
Let \((R_m)_{m\geq 1}\) be a sequence of positive integers and let
\((\delta_m)_{m\geq 1}\) satisfy \(0<\delta_m<1\) and
\(\delta_m\to0\).  Then there are sets \(U_m\in\cU\), with
\(U_{m+1}\subseteq U_m\), such that for every \(m\) and every \(n\in U_m\)
there is a set of components \(I_{n,m}\subseteq I_n\) satisfying
\[
        \sum_{i\in I_n\setminus I_{n,m}}
        \frac{|X_{n,i}|}{|X_n|}
        \leq \delta_m,
\]
and, for every \(i\in I_{n,m}\),
at least \((1-\delta_m)|X_{n,i}|\) vertices \(x\in X_{n,i}\) have rooted
\(S\)-labelled \(R_m\)-ball isomorphic to the rooted radius-\(R_m\) ball in
\(\operatorname{Cay}(G,S)\).  Consequently, if \(w\in F\) has length at most
\(R_m\) and represents \(1\in G\), then
\[
        \frac{|\{x\in X_{n,i}:\alpha_n(w)x\neq x\}|}{|X_{n,i}|}
        \leq \delta_m
        \qquad(i\in I_{n,m}).
\]
\end{lemma}

\begin{proof}
Fix \(m\).  Let
\[
        B_{n,m}
        =
        \{x\in X_n:
        B_{R_m}(x)\not\cong B_{R_m}(e,\operatorname{Cay}(G,S))\}.
\]
Since \((\alpha_n)_n\) is a sofic model for \(G\),
\(|B_{n,m}|/|X_n|\to_{\cU}0.\)
Set
$
        V_m=
        \left\{
        n:\ |B_{n,m}|\leq \delta_m^2 |X_n|
        \right\}\in\cU,
$
and replace \(V_m\) by
\(U_m=\bigcap_{\ell\leq m}V_\ell.\)
Then \(U_m\in\cU\) and \(U_{m+1}\subseteq U_m\).

For \(n\in U_m\), define
$
        I_{n,m}
        =
        \left\{
        i\in I_n:
        |B_{n,m}\cap X_{n,i}|
        \leq \delta_m |X_{n,i}|
        \right\}.
$
Then
\[
        \delta_m
        \sum_{i\notin I_{n,m}}|X_{n,i}|
        <
        \sum_{i\notin I_{n,m}}|B_{n,m}\cap X_{n,i}|
        \leq |B_{n,m}|
        \leq \delta_m^2 |X_n|.
\]
Dividing by \(\delta_m|X_n|\) gives the component-size estimate.

If \(w=1\) in \(G\) and \(|w|\leq R_m\), then \(\alpha_n(w)x=x\) at every
vertex whose \(R_m\)-ball is the Cayley \(R_m\)-ball.  Hence the failure set
of \(w\) on \(X_{n,i}\) is contained in \(B_{n,m}\cap X_{n,i}\).
\end{proof}

We use the following two-component form of the construction of Kun and the
second author, see \cite{KunThom}.

\begin{proposition}
\label{prop:two-component-repair}
There are constants \(C_1>0\) and \(\rho>0\), depending only on
\(G,S,h\), with the following property.  For
every \(0<\varepsilon<1\) and \(0<\zeta<1\), there are
\(R\in\mathbb N\) and \(0<\delta<1\) such that whenever \(Y,Z\) are finite
regularly \(S\)-labelled graphs with Cheeger constants at least \(h\), and at
least \((1-\delta)|Y|\), respectively \((1-\delta)|Z|\), vertices have correct
rooted \(R\)-balls for the Cayley graph of \(G\), every partial bijection
\(b:Y\dashrightarrow Z\) satisfying
\[
        \lambda_Y(b)\leq\rho,\qquad
        \lambda_Z(b)\leq\rho,\qquad
        \Delta_S(b)\leq\rho
\]
admits a partial bijection \(b':Y\dashrightarrow Z\) such that
\[
        \lambda_Y(b')\leq\varepsilon,\qquad
        \lambda_Z(b')\leq\varepsilon,\qquad
        \Delta_S(b')\leq\varepsilon,
\]
and
$
        d_Y(b,b')\leq C_1\Delta_S(b)+\zeta .
$
\end{proposition}

\begin{proof}
We need \cite[Proposition~3.3]{KunThom}: for every \(\alpha>0\) there is a radius \(r\) such that, in
any finite regularly \(S\)-labelled graph, a subset \(T\) all of whose points
have correct rooted \(r\)-balls is within
\(c|\partial_S T|\) of a subset \(U\) with
\(|\partial_S U|\leq \alpha |U|\).  Here \(c\) depends only on the Kazhdan
constant and on \(|S|\).

Fix \(\varepsilon\) and \(\zeta\).  Choose \(\alpha>0\) so small that every
constant multiple of \(\alpha/h\) occurring below is at most
\(\min(\varepsilon,\zeta)/20\), and let \(r\) be supplied by the aforementioned proposition for this \(\alpha\).  Let \(R\geq 2r+2\), and later choose
\(\delta>0\) sufficiently small.  Finally choose \(\rho>0\), depending only on
\(S,h\) and the constant \(c\), so small that
\(c\rho<1/100\) and all estimates below in which \(\rho\) is used to
compare \(|Y|\) and \(|Z|\) are valid.  Since
\(|\operatorname{dom}(b)|=|\operatorname{ran}(b)|\) and \(\lambda_Y(b),\lambda_Z(b)\leq\rho\), we have
\[
        1-3\rho\leq \frac{|Y|}{|Z|}\leq 1+3\rho .
\]

Consider the diagonal \(S\)-labelled graph on \(Y\times Z\),
\(s(y,z)=(sy,sz)\), and the graph
\[
        B=\{(y,b(y)):y\in \operatorname{dom}(b)\}\subseteq Y\times Z .
\]
If \((y,b(y))\in B\) and \(s(y,b(y))\notin B\), then either \(sy\notin \operatorname{dom}(b)\)
or \(b(sy)\neq s b(y)\). Therefore
$
        |\partial_S B|\leq \Delta_S(b)|Y| .
$
Let \(Y_{\rm loc}\subseteq Y\) and \(Z_{\rm loc}\subseteq Z\) be the vertices
with correct rooted \(R\)-balls.  Put
$
        B_{\rm loc}=B\cap(Y_{\rm loc}\times Z_{\rm loc}).
$
Since \(b\) is injective and \(|Z|/|Y|\) is bounded by \(2\), for \(\delta\)
small enough we have
$
        |B\setminus B_{\rm loc}|\leq 3\delta |Y| .
$
At each point of \(B_{\rm loc}\), the rooted \(r\)-ball in the diagonal graph is
the rooted \(r\)-ball in the Cayley graph of \(G\): any relation of length at
most \(2r\) visible in the diagonal ball would be visible in each coordinate,
and the coordinate balls are correct because \(R\geq 2r+2\).
Moreover
\[
        |\partial_S B_{\rm loc}|
        \leq |\partial_S B|+2|S|\,|B\setminus B_{\rm loc}|
        \leq \Delta_S(b)|Y|+6|S|\delta |Y| .
\]
Applying \cite[Proposition~3.3]{KunThom} in the diagonal graph gives a set
\(U\subseteq Y\times Z\) such that
$
        |\partial_S U|\leq \alpha |U|
$
and
\[
        |U\triangle B_{\rm loc}|
        \leq c\bigl(\Delta_S(b)+6|S|\delta\bigr)|Y| .
\]
After increasing the constant and choosing \(\delta\) so that the
\(\delta\)-term is at most \(\zeta |Y|/20\), we may record this as
\[
        |U\triangle B|
        \leq C_0\Delta_S(b)|Y|+\frac{\zeta}{20}|Y| .        \tag{1}
\]
In particular \(|U|=(1+O(\rho+\zeta))|Y|\).  Moreover, since \(B\) is the graph of a
partial bijection, every row in \(\operatorname{dom}(b)\) and every column in \(\operatorname{ran}(b)\) meets
\(B\) exactly once; a row or column can have different multiplicity in \(U\)
only if it meets \(U\triangle B\).  Hence, by \((1)\) and the choice of
\(\rho\), the rows of \(U\)-multiplicity exactly one number at least
\((1-\rho-C_0\rho-\zeta/20)|Y|>|Y|/2\), and likewise for columns.

For \(y\in Y\) and \(z\in Z\) define the row and column multiplicities
\[
        m_1(y)=|U\cap(\{y\}\times Z)|,
        \qquad
        m_2(z)=|U\cap(Y\times\{z\})|.
\]
For the row multiplicities, exactly as in the last paragraph of the proof of
\cite[Theorem~3.4]{KunThom},
\[
        \sum_{k\geq0}
        |E_Y(m_1^{-1}([0,k]),m_1^{-1}([k+1,\infty)))|
        \leq |\partial_S U|.
\]
Since the Cheeger constant of \(Y\) is at least \(h\) and the set of rows of
multiplicity one has size \(>|Y|/2\), this implies
\[
        |m_1^{-1}(0)|+
        \sum_{k\geq2}(k-1)|m_1^{-1}(k)|
        \leq \frac{\alpha}{h}|U| .        \tag{2}
\]
The identical argument in the \(Z\)-coordinate gives
\[
        |m_2^{-1}(0)|+
        \sum_{k\geq2}(k-1)|m_2^{-1}(k)|
        \leq \frac{\alpha}{h}|U| .        \tag{3}
\]
Delete points from rows and columns of multiplicity at least \(2\), and add
points in empty rows and columns, to obtain the graph \(B'\) of a partial
bijection \(b':Y\dashrightarrow Z\).  Equations \((2)\) and \((3)\) show that
this changes at most \(C_2\alpha |Y|/h\) points.  Consequently, by our choice of
\(\alpha\),
\[
        \lambda_Y(b'),\lambda_Z(b')\leq \varepsilon/2 .        \tag{4}
\]
The diagonal boundary of \(B'\) is bounded by the boundary of \(U\) plus at most
\(|S|\) times the number of points changed in the last row-column correction.
Thus
\[
        |\partial_S B'|
        \leq \alpha |U|+C_3|S|\alpha |Y|/h
        \leq \varepsilon |Y| .        \tag{5}
\]
Since \(B'\) is the graph of \(b'\), the boundary estimate \((5)\), together
with the source defect in \((4)\), is precisely
\(\Delta_S(b')\leq\varepsilon\) after decreasing \(\alpha\) once more.

Finally, \((1)\) and the row-column correction give
\[
        |B\triangle B'|
        \leq C_0\Delta_S(b)|Y|+\frac{\zeta}{20}|Y|+\frac{C_2\alpha}{h}|Y|.
\]
Recall that, in the definition of \(d_Y\), every point \(y\notin \operatorname{dom}(b)\cap
\operatorname{dom}(b')\) counts as a disagreement; the points outside \(\operatorname{dom}(b)\cup \operatorname{dom}(b')\) are
not seen by \(|B\triangle B'|\), but their number is at most
\(\lambda_Y(b)|Y|\leq \Delta_S(b)|Y|\), since every \(y\notin \operatorname{dom}(b)\)
contributes all \(|S|\) pairs \((y,s)\) to the defect count of \(b\).
Therefore
\[
        d_Y(b,b')
        \leq
        \frac{|B\triangle B'|}{|Y|}+\lambda_Y(b)
        \leq
        (C_0+1)\Delta_S(b)+\frac{\zeta}{20}+\frac{C_2\alpha}{h}.
\]
By the choice of \(\alpha\), the last two terms are at most \(\zeta\).  Thus
\(d_Y(b,b')\leq C_1\Delta_S(b)+\zeta\) with \(C_1:=C_0+1\), a constant
depending only on
\(G,S,h\).  This proves the proposition.
\end{proof}

\begin{lemma}\label{lem:cluster-scales}
There is a sequence
$
        \varepsilon_n\to_{\cU}0
$
and component sets \(J_n\subseteq I_n\), such that
\[
        \beta_n:=
        \sum_{i\in I_n\setminus J_n}\frac{|X_{n,i}|}{|X_n|}
        \to_{\cU}0.
\]
Put
\[
        K:=2|S|+2,
        \qquad
        q_n:=\frac{2\varepsilon_n}{h},
        \qquad
        r_n:=(KC_1+1)\varepsilon_n.
\]
Then \(q_n,r_n\to_{\cU}0\), and, for every \(n\) in a \(\cU\)-large set,
the following hold for components in \(J_n\).

\begin{enumerate}[label=\textup{(\roman*)}]
\item If \(b:X_{n,i}\dashrightarrow X_{n,j}\) satisfies
\[
        \lambda_{X_{n,i}}(b)\leq K\varepsilon_n,
        \qquad
        \lambda_{X_{n,j}}(b)\leq K\varepsilon_n,
        \qquad
        \Delta_S(b)\leq K\varepsilon_n,
\]
then there is \(b':X_{n,i}\dashrightarrow X_{n,j}\) satisfying
\[
        \lambda_{X_{n,i}}(b')\leq \varepsilon_n,
        \qquad
        \lambda_{X_{n,j}}(b')\leq \varepsilon_n,
        \qquad
        \Delta_S(b')\leq \varepsilon_n,
\]
and
$
        d_i(b,b')\leq r_n.
$

\item If \(b,c:X_{n,i}\dashrightarrow X_{n,j}\) satisfy the three
\(\varepsilon_n\)-bounds, then
\[
        d_i(b,c)\leq q_n
        \qquad\text{or}\qquad
        d_i(b,c)\geq 1-q_n.
\]

\item The  ordinary composition before correction  of two composable partial bijections satisfying the
three \(\varepsilon_n\)-bounds, and the inverse of any such partial bijection,
satisfy the three \(K\varepsilon_n\)-bounds.
\end{enumerate}
\end{lemma}

\begin{proof}
Choose \(a_m\downarrow0\) with \(a_m<1/m\).  Apply
Proposition~\ref{prop:two-component-repair} with
\(\varepsilon=a_m\) and \(\zeta=a_m\), and let \(R_m,\delta_m\) be the
corresponding constants.  After decreasing \(\delta_m\), assume
\(\delta_m<1/m\).  Applying Lemma~\ref{lem:good-components} to these
\(R_m,\delta_m\), obtain decreasing \(\cU\)-large sets \(U_m\) and component
sets \(I_{n,m}\) for \(n\in U_m\).  For \(n\in U_1\), define
\[
        m(n)=\max\{m\leq n:n\in U_m\},
        \qquad
        J_n=I_{n,m(n)},
        \qquad
        \varepsilon_n=a_{m(n)}.
\]
For \(n\notin U_1\), put \(m(n)=1\), \(J_n=I_n\), and
\(\varepsilon_n=a_1\).
Then \(m(n)\to_{\cU}\infty\),
\(\beta_n\leq\delta_{m(n)}\to_{\cU}0\), and
\(\varepsilon_n,q_n,r_n\to_{\cU}0\).  Shrinking to a \(\cU\)-large set, we
also assume
$
        K\varepsilon_n\leq\rho.
$

For \textup{(i)}, Proposition~\ref{prop:two-component-repair} gives
$
        d_i(b,b')
        \leq C_1\Delta_S(b)+\varepsilon_n
        \leq (KC_1+1)\varepsilon_n
        =r_n.
$

For \textup{(ii)}, let
$
        A=A(b,c)
        =\{x\in \operatorname{dom}(b)\cap \operatorname{dom}(c):b(x)=c(x)\}.
$
If \(x\in A\), \(s\in S\), and \(sx\notin A\), then \((x,s)\) is counted
in the defect set of \(b\) or of \(c\).  Hence
$
        |\partial_S A|\leq 2\varepsilon_n|X_{n,i}|.
$
Expansion gives either
\[
        |A|\leq q_n|X_{n,i}|
        \qquad\text{or}\qquad
        |X_{n,i}\setminus A|\leq q_n|X_{n,i}|.
\]
Since
\[
        d_i(b,c)=1-\frac{|A|}{|X_{n,i}|},
\]
this is exactly the asserted dichotomy.

For \textup{(iii)}, first note that every partial bijection
\(b:X_{n,i}\dashrightarrow X_{n,j}\) satisfying the three
\(\varepsilon_n\)-bounds obeys
\[
        1-\varepsilon_n
        \leq \frac{|X_{n,j}|}{|X_{n,i}|}
        \leq \frac{1}{1-\varepsilon_n}.
\]
Let
$
        b:X_{n,i}\dashrightarrow X_{n,j} $ and $
        c:X_{n,j}\dashrightarrow X_{n,k}
$
be allowed.  The source complement of \(cb\) is contained in
$
        (X_{n,i}\setminus \operatorname{dom}(b))
        \cup b^{-1}(X_{n,j}\setminus \operatorname{dom}(c)),
$
and the analogous estimate holds for the range complement.  Using this size comparison, both normalized defects are at most
\[
        \varepsilon_n+\frac{\varepsilon_n}{1-\varepsilon_n}
        \leq 3\varepsilon_n
\]
on a \(\cU\)-large set.  Likewise, a pair can be bad for \(cb\) only if it is
bad for \(b\), or its \(b\)-image is bad for \(c\).  Thus
\[
        \Delta_S(cb)
        \leq
        \varepsilon_n+\frac{\varepsilon_n}{1-\varepsilon_n}
        \leq3\varepsilon_n.
\]
These bounds are at most \(K\varepsilon_n\).

For inverses, the source and range defects are interchanged.  If a pair
\((z,s)\in X_{n,j}\times S\) is counted by \(\Delta_S(b^{-1})\), then either
\(z\notin \operatorname{ran}(b)\) or \(sz\notin \operatorname{ran}(b)\), accounting for at most
\(2|S|\varepsilon_n|X_{n,j}|\) pairs, or the pair corresponds injectively
under \(z=b(y)\) to a pair counted by \(\Delta_S(b)\).  Hence
\[
        \Delta_S(b^{-1})
        \leq
        2|S|\varepsilon_n
        +\frac{|X_{n,i}|}{|X_{n,j}|}\Delta_S(b)
        \leq(2|S|+2)\varepsilon_n
        =K\varepsilon_n,
\]
where the last inequality again follows from the size comparison above.
This proves \textup{(iii)}.
\end{proof}

We now replace \(X_n\) by the invariant subset
\[
        X_n^{\mathrm{good}}
        =
        \bigsqcup_{i\in J_n}X_{n,i}.
\]
Since
\(|X_n\setminus X_n^{\mathrm{good}}|/|X_n|=\beta_n\to_{\cU}0,\)
this is a passage within the same essential-equivalence class.  We then rename
\(X_n^{\mathrm{good}}\) as \(X_n\), \(J_n\) as \(I_n\), and keep the notation
\(X_n=\bigsqcup_{i\in I_n}X_{n,i}\).

Replacing the finite models and auxiliary data on a \(\cU\)-small set of
indices does not change any metric ultraproduct.  We therefore assume henceforth, for every
\(n\), that the conclusions of Lemma~\ref{lem:cluster-scales} hold and that
$
        K\varepsilon_n\leq\rho,$ and $
        \varepsilon_n,q_n,r_n<\frac{1}{100}.
$

\section{The cluster groupoid}

\begin{definition}\label{def:cluster-groupoid}
For each \(n\), define a finite groupoid
$
        \cC_n\rightrightarrows I_n
$
as follows.  The objects are \(i\in I_n\).  A morphism \(i\to j\) is an
equivalence class of partial bijections
\(b:X_{n,i}\dashrightarrow X_{n,j}\) satisfying
\[
        \lambda_{X_{n,i}}(b)\leq\varepsilon_n,
        \qquad
        \lambda_{X_{n,j}}(b)\leq\varepsilon_n,
        \qquad
        \Delta_S(b)\leq\varepsilon_n,
\]
where
\[
        b\sim c
        \quad\Longleftrightarrow\quad
        d_i(b,c)\leq\frac15.
\]
If \(\gamma:i\to j\) is represented by \(b\), write \(\gamma=[b]\).

If \(\gamma=[b]:i\to j\) and \(\delta=[c]:j\to k\), choose an improvement
\(u:X_{n,i}\dashrightarrow X_{n,k}\) of \(cb\) as in
Lemma~\ref{lem:cluster-scales}\textup{(i)}, and define
$
        \delta\gamma=[u].
$
\end{definition}

\begin{lemma}\label{lem:cluster-groupoid}
Definition~\ref{def:cluster-groupoid} gives a finite groupoid.  Moreover, any
two equivalent allowed maps are at distance at most \(q_n\).
\end{lemma}

\begin{proof}
The distance functions satisfy the triangle inequality.  We will also use the
elementary estimate
\[
        d_i(cb,c'b')
        \leq
        d_i(b,b')
        +\frac{|X_{n,j}|}{|X_{n,i}|}d_j(c,c').
\]
for partial bijections
$
        b,b':X_{n,i}\dashrightarrow X_{n,j}$ and $
        c,c':X_{n,j}\dashrightarrow X_{n,k}.
$
Indeed, outside the disagreement set of \(b,b'\), the two inner maps have a
common value, and the remaining exceptional points inject into the
disagreement set of \(c,c'\).

By Lemma~\ref{lem:cluster-scales}\textup{(ii)}, maps at distance at most
\(1/5\) are in fact at distance at most \(q_n\).  Reflexivity follows from
\(d_i(b,b)=\lambda_{X_{n,i}}(b)\leq\varepsilon_n<1/5\), and symmetry is
immediate.  If \(b\sim c\) and \(c\sim d\), then
\(d_i(b,d)\leq2/5<1-q_n\); the distance gap therefore gives
\(d_i(b,d)\leq q_n<1/5\).  Thus \(\sim\) is an equivalence relation.

Finiteness is immediate.  To prove that composition is well-defined, take
equivalent allowed maps
$
        b,b':X_{n,i}\dashrightarrow X_{n,j}$ and $
        c,c':X_{n,j}\dashrightarrow X_{n,k},
$
and improvements \(u,u'\) of \(cb,c'b'\), respectively.  By the size comparison in Lemma~\ref{lem:cluster-scales} and the
composition estimate above,
\[
        d_i(u,u')
        \leq
        2r_n+q_n+\frac{q_n}{1-\varepsilon_n}
        <\frac12.
\]
Both \(u,u'\) are allowed, so the distance gap gives
\(d_i(u,u')\leq q_n\).  This also covers different choices of improvement.

The identity at \(i\) is represented by
\(\operatorname{id}_{X_{n,i}}\).  An improvement of the raw product of an
identity with an allowed map \(b\) lies within \(r_n<1/2\) of \(b\), hence in
the same cluster by the distance gap.

Let \(b:X_{n,i}\dashrightarrow X_{n,j}\) be allowed.  By
Lemma~\ref{lem:cluster-scales}\textup{(iii)}, \(b^{-1}\) can be improved to
an allowed map \(\widehat b\); define
\([\,b\,]^{-1}=[\widehat b]\).  This is independent of the representative.
Indeed, if \(b\sim c\), then the image under \(b\) of the agreement set of
\(b,c\) is an agreement set for their inverses, and therefore
\[
        d_j(b^{-1},c^{-1})
        \leq
        1-(1-q_n)\frac{|X_{n,i}|}{|X_{n,j}|}
        \leq
        1-(1-q_n)(1-\varepsilon_n)
        \leq q_n+\varepsilon_n.
\]
After improving both inverses, their distance is at most
\(2r_n+q_n+\varepsilon_n<1/2\), and the distance gap puts them in the same
cluster.

Let \(u\) be an improvement of \(\widehat b b\).  Since
\(d_j(\widehat b,b^{-1})\leq r_n\),
\[
        d_i(\widehat b b,\operatorname{id}_{X_{n,i}})
        \leq
        \varepsilon_n+\frac{r_n}{1-\varepsilon_n}.
\]
Consequently,
\[
        d_i(u,\operatorname{id}_{X_{n,i}})
        \leq
        r_n+\varepsilon_n+\frac{r_n}{1-\varepsilon_n}
        <\frac12.
\]
The gap gives \([\widehat b][b]=[\operatorname{id}_{X_{n,i}}]\).  The other
identity follows similarly; in fact
\(d_j(b\widehat b,\operatorname{id}_{X_{n,j}})\leq r_n\).

It remains to prove associativity.  Let
$
        [b]:i\to j,
        [c]:j\to k$ and $
        [d]:k\to\ell
$
be composable.  Choose improvements \(u\) of \(cb\), \(v\) of \(dc\), and
then improvements \(p\) of \(du\), \(q\) of \(vb\).  Applying the composition estimate twice gives
\[
        d_i(du,dcb)
        \leq
        r_n+\frac{\varepsilon_n}{(1-\varepsilon_n)^2}
\]
and
\[
        d_i(vb,dcb)
        \leq
        q_n+\frac{r_n}{1-\varepsilon_n}.
\]
Hence
\[
        d_i(p,q)
        \leq
        3r_n+q_n
        +\frac{r_n}{1-\varepsilon_n}
        +\frac{\varepsilon_n}{(1-\varepsilon_n)^2}
        <\frac12.
\]
The maps \(p,q\) are allowed, so the distance gap makes them equivalent.
Thus the product is associative.
\end{proof}

\begin{definition}\label{def:full-group}
The finite full group \([[\cC_n]]\) is the group of total bisections of
\(\cC_n\).  Thus \(a\in[[\cC_n]]\) consists of a permutation
\(\bar a:I_n\to I_n\) and, for every \(i\in I_n\), an arrow
\(a_i:i\to \bar a(i).\)
Multiplication is defined as
\((ab)_i=a_{\bar b(i)}\,b_i .\)
\end{definition}

For every arrow \(\gamma:i\to j\) of \(\cC_n\), fix an
allowed representative
$
        \theta_n(\gamma):X_{n,i}\dashrightarrow X_{n,j},
$
choosing the identity for identity arrows.  For
\(a\in[[\cC_n]]\), patch the maps \(\theta_n(a_i)\) over all source
components.  On every connected component \(\Omega\) of \(\cC_n\), the
patched map has equally large source and range complements inside
\[
        X_{n,\Omega}:=\bigsqcup_{i\in\Omega}X_{n,i}.
\]
Complete it separately on each \(X_{n,\Omega}\), and denote the resulting
permutation by
$
        \widetilde\rho_n(a)\in\Sym(X_n).
$

\subsection{Recovering the centralizer by bisections}

For \(\sigma_n\in\Sym(X_n)\), write
\[
        \operatorname{com}_S(\sigma_n)
        =
        \sum_{s\in S}
        d_H(\sigma_n\alpha_n(s),\alpha_n(s)\sigma_n).
\]
For \(A\subseteq I_n\), write
\[
        \omega_n(A)=\sum_{i\in A}\frac{|X_{n,i}|}{|X_n|}.
\]
The following two-sided majority argument is adapted from the good/bad-edge
argument in \cite[Section~5]{KunThom}.

\begin{lemma}\label{lem:component-nonsplitting}
Let \(\sigma_n\in\Sym(X_n)\), and put
$
        c_n=\operatorname{com}_S(\sigma_n),
        \delta_n=\frac{c_n}{h}.
$
Then there are a subset \(L_n\subseteq I_n\) and an injective map
\(\tau_n:L_n\to I_n\) such that, for
\[
        D_i=X_{n,i}\cap\sigma_n^{-1}(X_{n,\tau_n(i)}),
        \qquad
        R_i=\sigma_n(D_i),
        \qquad i\in L_n,
\]
one has
\[
        |D_i|>\frac{|X_{n,i}|}{2},
        \qquad
        |R_i|>\frac{|X_{n,\tau_n(i)}|}{2},
\]
and
\begin{align*}
        \omega_n(I_n\setminus L_n)&\leq 3\delta_n,\quad
        \sum_{i\in L_n}\frac{|X_{n,i}\setminus D_i|}{|X_n|}
        &\leq \delta_n, \quad
        \sum_{i\in L_n}
        \frac{|X_{n,\tau_n(i)}\setminus R_i|}{|X_n|}
        &\leq \delta_n.
\end{align*}
\end{lemma}

\begin{proof}
For \(i,j\in I_n\), set
\[
        P_{i,j}=X_{n,i}\cap\sigma_n^{-1}(X_{n,j}),
        \qquad
        Q_{j,i}=\sigma_n(P_{i,j})
        =X_{n,j}\cap\sigma_n(X_{n,i}).
\]
For fixed \(i\), the sets \(P_{i,j}\) partition \(X_{n,i}\), and for fixed
\(j\), the sets \(Q_{j,i}\) partition \(X_{n,j}\).

Let
$
        \partial_i A=\{(x,s)\in A\times S:sx\notin A\}
$
be the directed \(S\)-edge boundary inside \(X_{n,i}\).  If
\((x,s)\in\partial_iP_{i,j}\), then
\(\sigma_n(sx)\neq s\sigma_n(x)\).  Hence
\[
        \sum_{i,j}|\partial_iP_{i,j}|
        \leq c_n|X_n|.
\]
Define the forward-bad set
\[
        B_n^+
        =
        \bigsqcup_{\substack{i,j\in I_n\\
        |P_{i,j}|\leq |X_{n,i}|/2}}P_{i,j}.
\]
Each cell in this union has size at most half of its component, so expansion
and the preceding boundary estimate give
\[
        h|B_n^+|
        \leq
        \sum_{\substack{i,j\in I_n\\
        |P_{i,j}|\leq |X_{n,i}|/2}}
        |\partial_iP_{i,j}|
        \leq c_n|X_n|.
\]
Thus \(|B_n^+|\leq\delta_n|X_n|\).

Apply the same argument to \(\sigma_n^{-1}\).  Since
\(\operatorname{com}_S(\sigma_n^{-1})=
\operatorname{com}_S(\sigma_n)\), the corresponding backward-bad set
\[
        B_n^-
        =
        \bigsqcup_{\substack{j,i\in I_n\\
        |Q_{j,i}|\leq |X_{n,j}|/2}}Q_{j,i}
\]
satisfies \(|B_n^-|\leq\delta_n|X_n|\).

Let \(L_n\) consist of those \(i\in I_n\) for which there is a \(j\in I_n\)
such that
\[
        |P_{i,j}|>\frac{|X_{n,i}|}{2}
        \quad\text{and}\quad
        |Q_{j,i}|>\frac{|X_{n,j}|}{2},
\]
and denote this unique \(j\) by \(\tau_n(i)\).  The map \(\tau_n\) is
injective: if \(\tau_n(i)=\tau_n(i')=j\), then the disjoint sets
\(Q_{j,i}\) and \(Q_{j,i'}\) both have size greater than \(|X_{n,j}|/2\),
so \(i=i'\).

If no cell \(P_{i,j}\) has size greater than \(|X_{n,i}|/2\), then
\(X_{n,i}\subseteq B_n^+\).  For every remaining \(i\notin L_n\), let \(j\)
be its unique forward-majority target.  Then \(Q_{j,i}\subseteq B_n^-\) and
$
        |X_{n,i}|<2|P_{i,j}|=2|Q_{j,i}|.
$
Since the cells \(Q_{j,i}\) are pairwise disjoint, summing these two cases
gives
\[
        \sum_{i\in I_n\setminus L_n}|X_{n,i}|
        \leq |B_n^+|+2|B_n^-|
        \leq 3\delta_n|X_n|.
\]

For \(i\in L_n\), every cell \(P_{i,j}\) with \(j\neq\tau_n(i)\) is
forward-bad.  Therefore
\[
        \bigsqcup_{i\in L_n}(X_{n,i}\setminus D_i)
        \subseteq B_n^+,
\]
which gives the second estimate.  Similarly, every cell
\(Q_{\tau_n(i),k}\) with \(k\neq i\) is backward-bad.  Since \(\tau_n\) is
injective,
\[
        \bigsqcup_{i\in L_n}(X_{n,\tau_n(i)}\setminus R_i)
        \subseteq B_n^-,
\]
which gives the third estimate.
\end{proof}

\begin{proposition}\label{prop:metric-groupoid-recovery}
For the fixed maps
\(\widetilde\rho_n:[[\cC_n]]\to\Sym(X_n)\) defined above, the following
hold.

\begin{enumerate}[label=\textup{(\alph*)}]
\item
\[
        \sup_{a,b\in[[\cC_n]]}
        d_H\bigl(
        \widetilde\rho_n(ab),
        \widetilde\rho_n(a)\widetilde\rho_n(b)
        \bigr)
        \leq q_n+r_n.
\]

\item
\[
        \sup_{a\in[[\cC_n]]}
        \sum_{s\in S}
        d_H\bigl(
        \widetilde\rho_n(a)\alpha_n(s),
        \alpha_n(s)\widetilde\rho_n(a)
        \bigr)
        \leq\varepsilon_n.
\]

\item If \(\sigma=[\sigma_n]_{\cU}\in C(\pi(G))\), then there are
\(a_n\in[[\cC_n]]\) such that
\[
        d_H\bigl(\sigma_n,\widetilde\rho_n(a_n)\bigr)
        \to_{\cU}0.
\]
\end{enumerate}
Consequently, the centralizer consists exactly of the ultraproduct classes
represented by sequences \(\widetilde\rho_n(a_n)\),
\(a_n\in[[\cC_n]]\).
\end{proposition}

\begin{proof}
We first prove \textup{(a)}.  Let \(a,b\in[[\cC_n]]\).  On the component
\(X_{n,i}\), let \(u_i\) be an improvement of the ordinary composition before correction
$
        \theta_n(a_{\bar b(i)})\theta_n(b_i).
$
By the definition of multiplication in \(\cC_n\), both \(u_i\) and
\(\theta_n((ab)_i)\) represent the arrow \((ab)_i\).  Hence
Lemma~\ref{lem:cluster-groupoid} and
Lemma~\ref{lem:cluster-scales}\textup{(i)} give
\[
\begin{aligned}
        d_i\bigl(
        \theta_n((ab)_i),
        \theta_n(a_{\bar b(i)})\theta_n(b_i)
        \bigr)
        &\leq
        d_i(\theta_n((ab)_i),u_i)
        +d_i\bigl(u_i,
        \theta_n(a_{\bar b(i)})\theta_n(b_i)\bigr)
        \leq q_n+r_n.
\end{aligned}
\]
Whenever the two partial maps in this comparison are defined and equal at
\(x\), the chosen permutation completions are also equal at \(x\).
Therefore no additional completion term is needed, and summing over the
components proves \textup{(a)}.

For \textup{(b)}, write \(\bar a:I_n\to I_n\) for the object permutation of
\(a\).  If
$
        x,\alpha_n(s)x\in D(\theta_n(a_i))
$
and
$
        \theta_n(a_i)(\alpha_n(s)x)
        =\alpha_n(s)\theta_n(a_i)(x),
$
then the completion \(\widetilde\rho_n(a)\) commutes with \(\alpha_n(s)\)
at \(x\).  Hence
\[
\begin{aligned}
        \sum_{s\in S}
        d_H\bigl(
        \widetilde\rho_n(a)\alpha_n(s),
        \alpha_n(s)\widetilde\rho_n(a)
        \bigr)
        &\leq
        \sum_{i\in I_n}
        \frac{|X_{n,i}|}{|X_n|}
        \Delta_S(\theta_n(a_i))
        \leq\varepsilon_n.
\end{aligned}
\]

We prove \textup{(c)}.  Put
$
        c_n=\operatorname{com}_S(\sigma_n)$ and $
        \delta_n=\frac{c_n}{h}.
$
Then \(c_n,\delta_n\to_{\cU}0\).  Apply
Lemma~\ref{lem:component-nonsplitting}, retaining its notation
\(L_n,\tau_n,D_i,R_i\), and write
$
        \sigma_{n,i}:D_i\to R_i
$
for the restriction of \(\sigma_n\).

If a pair \((x,s)\) is counted by \(\Delta_S(\sigma_{n,i})\) and
\(x\in D_i\), then it is an actual commutator defect.  Indeed, if
\(sx\notin D_i\), then \(s\sigma_n(x)\) lies in
\(X_{n,\tau_n(i)}\), whereas \(\sigma_n(sx)\) does not.  Consequently,
\[
\begin{aligned}
        \sum_{i\in L_n}
        \frac{|X_{n,i}|}{|X_n|}\Delta_S(\sigma_{n,i})
        &\leq
        c_n+|S|
        \sum_{i\in L_n}
        \frac{|X_{n,i}\setminus D_i|}{|X_n|}
        \leq c_n+|S|\delta_n.
\end{aligned}
\]
Put
$
        \Xi_n:=c_n+|S|\delta_n.
$

Let \(K_n\subseteq L_n\) consist of those \(i\) for which
$
        \lambda_{X_{n,\tau_n(i)}}(\sigma_{n,i})\leq\rho$ and $
        \Delta_S(\sigma_{n,i})\leq\rho.
$
The source-defect hypothesis in Proposition~\ref{prop:two-component-repair}
is automatic here: every point outside \(D_i\) contributes all \(|S|\)
pairs based at that point, so
$
        \lambda_{X_{n,i}}(\sigma_{n,i})
        \leq\Delta_S(\sigma_{n,i}).
$
Since \(D_i\) and \(R_i\) are strict majorities of equal cardinality,
$
        |X_{n,i}|<2|X_{n,\tau_n(i)}|.
$
The range-defect estimate and Markov's inequality therefore give
\[
        \omega_n\bigl(\{i\in L_n:
        \lambda_{X_{n,\tau_n(i)}}(\sigma_{n,i})>\rho\}\bigr)
        \leq\frac{2\delta_n}{\rho},
\]
while
\[
        \omega_n\bigl(\{i\in L_n:
        \Delta_S(\sigma_{n,i})>\rho\}\bigr)
        \leq\frac{\Xi_n}{\rho}.
\]
Together with Lemma~\ref{lem:component-nonsplitting}, this yields
\[
        \omega_n(I_n\setminus K_n)
        \leq
        3\delta_n+\frac{2\delta_n}{\rho}+\frac{\Xi_n}{\rho}
        \to_{\cU}0.
\]

For \(i\in K_n\), apply
Proposition~\ref{prop:two-component-repair} with
\(\varepsilon=\zeta=\varepsilon_n\).  The choice of the good components in
Lemma~\ref{lem:cluster-scales} ensures that its local hypotheses hold.  We
obtain an allowed partial bijection
$
        b_{n,i}:X_{n,i}\dashrightarrow X_{n,\tau_n(i)}
$
such that
$
        d_i(b_{n,i},\sigma_{n,i})
        \leq C_1\Delta_S(\sigma_{n,i})+\varepsilon_n.
$
Let \(\gamma_{n,i}:i\to\tau_n(i)\) be its arrow in \(\cC_n\).  The arrows
\(\gamma_{n,i}\), \(i\in K_n\), form a partial bisection.

For each connected component \(\Omega\) of \(\cC_n\), choose a bijection
\[
        \rho_{n,\Omega}:
        \Omega\setminus(K_n\cap\Omega)
        \longrightarrow
        \Omega\setminus\tau_n(K_n\cap\Omega).
\]
This is possible because \(\tau_n\) is injective.  For
\(i\in\Omega\setminus K_n\), choose an arrow
\(\delta_{n,i}:i\to\rho_{n,\Omega}(i)\), and define a total bisection
\(a_n\in[[\cC_n]]\) by
\[
        (a_n)_i=
        \begin{cases}
        \gamma_{n,i},&i\in K_n,\\
        \delta_{n,i},&i\notin K_n.
        \end{cases}
\]

For \(i\in K_n\), the maps \(b_{n,i}\) and
\(\theta_n(\gamma_{n,i})\) represent the same arrow, so
Lemma~\ref{lem:cluster-groupoid} gives
$
        d_i\bigl(b_{n,i},\theta_n(\gamma_{n,i})\bigr)\leq q_n.
$
As before, undefined points are already counted by \(d_i\), so the arbitrary
completion used in \(\widetilde\rho_n(a_n)\) contributes no further term.
Thus
\[
\begin{aligned}
        d_H\bigl(\sigma_n,\widetilde\rho_n(a_n)\bigr)
        &\leq
        \omega_n(I_n\setminus K_n)+
        \sum_{i\in K_n}\frac{|X_{n,i}|}{|X_n|}
        \left(
        d_i(\sigma_{n,i},b_{n,i})
        +d_i\bigl(b_{n,i},\theta_n(\gamma_{n,i})\bigr)
        \right)\\
        &\leq
        \omega_n(I_n\setminus K_n)
        +C_1\Xi_n+\varepsilon_n+q_n
        \to_{\cU}0.
\end{aligned}
\]
This proves \textup{(c)} and the final assertion.
\end{proof}

\begin{corollary}\label{cor:conull-groupoid-component}
Assume that \(C(\pi(G))\) acts ergodically on the associated Loeb probability
space.  Then there are connected components \(\Omega_n\subseteq I_n\) of
\(\cC_n\) such that
$
        \omega_n(\Omega_n)\to_{\cU}1.
$
Moreover, for every \(i,j\in\Omega_n\),
\[
        1-\varepsilon_n
        \leq\frac{|X_{n,i}|}{|X_{n,j}|}
        \leq\frac{1}{1-\varepsilon_n}.
\]
\end{corollary}

\begin{proof}
Let \(A_n\subseteq I_n\) be any union of connected components of \(\cC_n\),
and put
\[
        Z_n=\bigsqcup_{i\in A_n}X_{n,i}.
\]
Every total bisection preserves \(A_n\) at the object level, and the
componentwise completion chosen above therefore preserves \(Z_n\) exactly.
Proposition~\ref{prop:metric-groupoid-recovery} then shows that
\([Z_n]_{\cU}\) is invariant under the whole centralizer.
Ergodicity implies that every ultralimit of such unions has measure \(0\) or
\(1\).

If no sequence of connected components had weight tending to \(1\), there
would be \(\delta>0\) and a \(\cU\)-large set of indices on which every
connected component has weight at most \(1-\delta\).  Therefore there are unions \(A_n\) of connected components with
\[
        \frac{\delta}{2}
        \leq\omega_n(A_n)
        \leq1-\frac{\delta}{2},
\]
contradicting ergodicity.  This proves the first assertion.

If \(i,j\in\Omega_n\), we have an arrow
\(i\to j\).  An allowed representative of that arrow has a domain of size at
least \((1-\varepsilon_n)|X_{n,i}|\) and a range of the same size, missing at
most \(\varepsilon_n|X_{n,j}|\).  The stated size comparison follows.
\end{proof}

\begin{remark}
For a connected component \(\Omega\) of \(\cC_n\), choosing a base object
\(o\in\Omega\) and arrows from \(o\) to the other objects gives a noncanonical
isomorphism
\[
        [[\cC_n|_{\Omega}]]
        \cong
        \cC_n(o,o)^{\Omega}\rtimes\Sym(\Omega).
\]
Thus each connected factor of the finite full group is a wreath product of a
isotropy group with a symmetric group.
\end{remark}

\section{Stability and proof of the main theorem}

It remains to correct the approximate actions to genuine actions
\(\widetilde\rho_n:[[\cC_n]]\to\Sym(X_n)\) from
Proposition~\ref{prop:metric-groupoid-recovery}.  We use Becker--Chapman
flexible stability \cite{BeckerChapman} in the following uniform variant: for
every \(\varepsilon>0\) there exists \(\delta>0\) such that, for \emph{every}
finite group \(\Gamma\) and every map \(f:\Gamma\to\Sym(V)\) with
$
        \sup_{a,b\in\Gamma}
        d_H(f(ab),f(a)f(b))\leq\delta,
$
there are a finite set \(V'\supseteq V\) with
\(|V'\setminus V|\leq\varepsilon|V'|\) and a homomorphism
\(\bar f:\Gamma\to\Sym(V')\) such that, after extending \(f(a)\) by the
identity on \(V'\setminus V\),
$
        \sup_{a\in\Gamma} d_H(\bar f(a),f(a))\leq\varepsilon.
$
Applied to a sequence
\(\Gamma_m\) of finite groups and maps \(f_m:\Gamma_m\to\Sym(V_m)\) with
$$\sup_{a,b\in\Gamma_m}d_H(f_m(ab),f_m(a)f_m(b))\to0,$$ it produces finite
sets \(V'_m\supseteq V_m\) with \(|V'_m\setminus V_m|/|V'_m|\to0\) and
homomorphisms \(\bar f_m:\Gamma_m\to\Sym(V'_m)\) with
\(\sup_{a\in\Gamma_m}d_H(\bar f_m(a),f_m(a))\to0\); the same holds with
ordinary limits replaced by limits along an ultrafilter.

\begin{proposition}\label{prop:exactification}
Let
\(F_n=[[\cC_n]].\)
There are finite sets \(Y_n\supseteq X_n\) with
\[
        \frac{|Y_n\setminus X_n|}{|Y_n|}\to_{\cU}0
\]
and homomorphisms
\(\rho_n:F_n\to\Sym(Y_n)\)
such that, after extending each \(\alpha_n(s)\) by the identity on
\(Y_n\setminus X_n\),
\(C(\pi(G))=\prod_{\cU}\rho_n(F_n)\)
inside \(\prod_{\cU}\Sym(Y_n)\).
\end{proposition}

\begin{proof}
By Proposition~\ref{prop:metric-groupoid-recovery}\textup{(a)},
\[
        \operatorname{def}_n
        :=
        \sup_{a,b\in F_n}
        d_H\bigl(
        \widetilde\rho_n(ab),
        \widetilde\rho_n(a)\widetilde\rho_n(b)
        \bigr)
        \leq q_n+r_n
        \to_{\cU}0.
\]
Applying the
Becker--Chapman theorem stated above to
\(\widetilde\rho_n:F_n\to\Sym(X_n)\) gives a finite set
\(Y_n\supseteq X_n\) and a homomorphism
\(\rho_n:F_n\to\Sym(Y_n)\)
such that, extending \(\widetilde\rho_n(a)\) by the identity on
\(Y_n\setminus X_n\),
$
        \sup_{a\in F_n}
        d_H(\rho_n(a),\widetilde\rho_n(a))\to_{\cU}0
$
and
\[
        \frac{|Y_n\setminus X_n|}{|Y_n|}\to_{\cU}0.
\]

Extend the \(F\)-model \(\alpha_n\) to \(Y_n\) by fixing the added points.  For
\(a\in F_n\) and \(s\in S\),
\[
\begin{aligned}
        d_H(\rho_n(a)\alpha_n(s),\alpha_n(s)\rho_n(a))
        &\leq
        2d_H(\rho_n(a),\widetilde\rho_n(a)) 
        +
        d_H(\widetilde\rho_n(a)\alpha_n(s),
        \alpha_n(s)\widetilde\rho_n(a)).
\end{aligned}
\]
Taking the supremum over \(a\in F_n\), summing over \(s\in S\), and using
Proposition~\ref{prop:metric-groupoid-recovery}\textup{(b)}, we get
\[
        \sup_{a\in F_n}
        \sum_{s\in S}
        d_H(\rho_n(a)\alpha_n(s),\alpha_n(s)\rho_n(a))
        \to_{\cU}0.
\]
Thus
\(\prod_{\cU}\rho_n(F_n)\subseteq C(\pi(G)).\)

Conversely, let \(\sigma=[\sigma_n]_{\cU}\in C(\pi(G))\) be represented by
permutations of \(Y_n\).  Since \(Y_n\setminus X_n\) has vanishing measure,
changing \(\sigma_n\) on
\((Y_n\setminus X_n)\cup\sigma_n^{-1}(Y_n\setminus X_n)\)
changes \(\sigma_n\) on at most \(2|Y_n\setminus X_n|\) points and gives an
equivalent representative preserving \(X_n\).  After changing \(\sigma_n\)
on the added points once more, we may also assume
that it is the identity on \(Y_n\setminus X_n\).  Restricting to \(X_n\),
Proposition~\ref{prop:metric-groupoid-recovery}\textup{(c)} gives
\(a_n\in F_n\) such that
$
        d_H\bigl(\sigma_n,\widetilde\rho_n(a_n)\bigr)
        \to_{\cU}0.
$
Since \(\rho_n(a_n)\) is uniformly close to \(\widetilde\rho_n(a_n)\), we get
\(d_H(\sigma_n,\rho_n(a_n))\to_{\cU}0.\)
Hence \(\sigma\in\prod_{\cU}\rho_n(F_n)\); this proves equality.
Because \(|Y_n\setminus X_n|/|Y_n|\to_{\cU}0\), enlarging from \(X_n\) to
\(Y_n\) does not change the represented ultraproduct element or the associated
Loeb action.
\end{proof}

\begin{proof}[Proof of Theorem~\ref{thm:groupoid-recovery}]
Start with the maps
\(\alpha_n:F_S\to\Sym(X_n)\) of the sofic embedding \(\pi\).  First
replace \(X_n\) by a union of expander components supplied by
Kun's expander decomposition; this removes a set of measure tending to \(0\)
along \(\cU\).  Next, Lemmas~\ref{lem:good-components} and
\ref{lem:cluster-scales} give component sets \(J_n\) with complement weight
\(\beta_n\to_{\cU}0\).  Replace \(X_n\) by
\(\bigsqcup_{i\in J_n}X_{n,i}.\)
After this passage within the essential-equivalence class, all components are good and the cluster groupoids
$
        \cC_n\rightrightarrows I_n
$
are defined by Definition~\ref{def:cluster-groupoid}.  They are finite
groupoids by Lemma~\ref{lem:cluster-groupoid}.

Proposition~\ref{prop:metric-groupoid-recovery} turns the finite full groups
\(F_n=[[\cC_n]]\) into uniform approximate permutation actions which
asymptotically commute with the \(F_S\)-model and exhaust the centralizer.
Proposition~\ref{prop:exactification} then replaces \(X_n\) by supersets
\(Y_n\supseteq X_n\), with added measure tending to \(0\), and gives genuine
homomorphisms
\(\rho_n:F_n\to\Sym(Y_n)\)
such that in this new sofic approximation $\pi':G\to\prod_{\cU}\Sym(Y_n)$ we have
\(C(\pi'(G))=\prod_{\cU}\rho_n(F_n).\)
Setting
\(
        A_n=\rho_n(F_n)\leq\Sym(Y_n)
\)
finishes the proof.
\end{proof}

We are now ready to prove the main theorem.

\begin{proof}[Proof of Theorem~\ref{thm:main}]
Let \(G\) be finitely generated and Kazhdan, and let
\(\pi:G\to \prod_{\cU}\Sym(X_n)\) be a sofic embedding whose centralizer acts
ergodically on the associated Loeb probability space.

By Theorem~\ref{thm:groupoid-recovery}, applied to the expander component
decomposition of finite free-group representatives of \(\pi\), there exist an
equivalent sofic embedding
\(\pi':G\to \prod_{\cU}\Sym(Y_n)\), obtained from \(\pi\) by the
replacements listed in Theorem~\ref{thm:groupoid-recovery}, and finite permutation
groups \(A_n\leq \Sym(Y_n)\) such that
$
        C_{\cS_{\mathcal U}}(\pi'(G))=\prod_{\cU}A_n
$
inside \(\prod_{\cU}\Sym(Y_n)\).

By Lemma~\ref{lem:null-modification}\textup{(ii)}, each of these replacements
induces an isometric isomorphism of the ambient metric ultraproducts which
carries \(\pi(G)\) to \(\pi'(G)\), carries centralizer to centralizer, and
identifies the associated Loeb probability spaces equivariantly modulo null
sets. Hence the centralizer of \(\pi'(G)\) still acts ergodically. Therefore
Proposition~\ref{prop:algebraic-recovery-rf}, applied to \(\pi'\), gives that
\(G\) is LEF. If \(G\) is finitely presented, the final assertion of
Proposition~\ref{prop:algebraic-recovery-rf} gives that \(G\) is residually
finite.
\end{proof}

Theorem~\ref{thm:main} shows that the Hayes--Kunnawalkam Elayavalli
conjecture has the following consequence.

\begin{corollary}
If a finitely generated Kazhdan sofic group admits an embedding as in the
Hayes--Kunnawalkam Elayavalli conjecture, then it is LEF.  In particular,
every such finitely presented group is residually finite.
\end{corollary}

\begin{proof}
The assumed embedding has ergodic centralizer, so
Theorem~\ref{thm:main} applies.
\end{proof}

This conclusion for finitely presented Kazhdan groups is
compatible with the known examples, in the sense that it neither contradicts
the conjecture above nor proves the existence of a non-sofic group.

Let's review some of the known examples:
First, the second author constructed a Kazhdan group which is locally
embeddable into finite groups, hence sofic, but not residually finite
\cite{ThomHyperlinear}.
Since a discrete Kazhdan group is finitely generated, this gives a finitely
generated sofic Kazhdan group which is not residually finite.  It is not
finitely presented: every finitely presented LEF group is residually finite,
because the finite presentation turns sufficiently good local embeddings into
genuine finite homomorphisms.
Second, de Cornulier constructed finitely presentable non-Hopfian Kazhdan
groups \cite{deCornulier}.  Since finitely generated residually finite groups
are Hopfian, these groups are not residually finite.  Related examples are
known to be hyperlinear by work of the second author \cite{ThomHyperlinear},
but they are not known to be sofic. 
Third, Kar and Nikolov constructed finitely presented sofic groups which are
not residually finite \cite{KarNikolov}.  These
examples do not have Kazhdan's property.

\section{Open problems}

We formulate the following two open problems.

\begin{openproblem}
Are there examples of finitely presented sofic groups with Kazhdan's property~\((T)\) that are not residually finite?
\end{openproblem}

\begin{openproblem}
Let \(G\) be a Kazhdan group and let $S$ be a finite symmetric generating set.  Let
$
\pi\colon G \longrightarrow
\mathcal U\!\left(
  \prod_{\cU}
  M_{d_n}(\mathbb C)
\right)
$
be a homomorphism into the unitary group of a tracial ultraproduct of matrix algebras. 
\begin{enumerate}
\item[(a)] After replacing $M_{d_n}(\mathbb C)$ by $M_{m_n}(\mathbb C)$ for a suitable sequence $(m_n)_n$ with
\(m_n/d_n\to_{\cU}1\), must there exist finite-dimensional 
$*$-subalgebras \(A_n\subseteq M_{m_n}(\mathbb C)\) such that,
after the canonical identification $\prod_{\cU} M_{d_n}(\mathbb C) = \prod_{\cU} M_{m_n}(\mathbb C)$, we have
\[
{\pi}(G)'
\cap \prod_{\cU} M_{m_n}(\mathbb C)
=
\prod_{\cU} A_n?
\]
\item[(b)]  Can this be done by finding lifts of $\pi(s)$ for $s \in S$ to $(g_{s,n})_{n} \in \prod_{\mathbb N} U(m_n)$ such that $A_n$ can be defined as the centralizer of $\{g_{s,n} : s \in S\}$ in $M_{m_n}(\mathbb C)$?
\end{enumerate}
\end{openproblem}

\section*{Acknowledgements}

AI assistants (ChatGPT by OpenAI and Claude by Anthropic) were used to assist
with drafting, editing, and checking parts of preliminary versions of this manuscript. All
mathematical content was reviewed, checked, and substantially revised by the
authors, who are responsible for the final text.

\end{document}